\documentclass[a4paper,12pt]{article}

\usepackage[left=2cm,right=2cm, top=2cm,bottom=3cm,bindingoffset=0cm]{geometry}

\usepackage{verbatim}
\usepackage{amsmath}
\usepackage{amsthm}
\usepackage{amssymb}
\usepackage{delarray}
\usepackage{cite}
\usepackage{hyperref}
\usepackage{mathrsfs}
\usepackage{tikz}
\usetikzlibrary{patterns}
\usepackage{caption}
\DeclareCaptionLabelSeparator{dot}{. }
\newcommand{\al}{\alpha}

\newcommand{\ga}{\gamma}
\newcommand{\de}{\delta}
\newcommand{\la}{\lambda}

\newcommand{\eps}{\varepsilon}
\newcommand{\vv}{\varphi}
\newcommand{\iy}{\infty}

\theoremstyle{plain}

\numberwithin{equation}{section}

\newtheorem{thm}{Theorem}[section]
\newtheorem{lem}[thm]{Lemma}
\newtheorem{prop}[thm]{Proposition}

\theoremstyle{definition}

\newtheorem{alg}[thm]{Algorithm}
\newtheorem{ip}[thm]{IP}
\newtheorem{df}[thm]{Definition}

\theoremstyle{remark}

 \allowdisplaybreaks

\begin{document}

\begin{center}
{\Large\bf Inverse problem for a differential operator on a star-shaped graph with nonlocal matching condition}
\\[0.2cm]
{\bf Natalia P. Bondarenko} \\[0.2cm]
\end{center}

\vspace{0.5cm}

{\bf Abstract.} In this paper, we develop two approaches to investigation of inverse spectral problems for a new class of nonlocal operators on metric graphs. The Laplace differential operator is considered on a star-shaped graph with nonlocal integral matching condition. This operator is adjoint to the functional-differential operator with frozen argument at the central vertex of the graph. We study the inverse problem that consists in the recovery of the integral condition coefficients from the eigenvalues. We obtain the spectrum characterization, reconstruction algorithms, and prove the uniqueness of the inverse problem solution. 

\medskip


\medskip


\vspace{1cm}

\section{Introduction} \label{sec:intr}

Consider a star-shaped graph having $m \ge 2$ edges of equal length $\pi$. All the edges have a common central vertex.
This paper deals with the following eigenvalue problem $L$ for the Laplace equations on the described graph:
\begin{equation} \label{eqv}
    -y_j''(x) = \lambda y_j(x), \quad x \in (0, \pi), \quad j = \overline{1, m},
\end{equation}
with the Robin boundary conditions (BCs)
\begin{equation}
    \label{bc}
    y_j'(0) - h_j y_j(0) = 0, \quad j = \overline{1, m}, 
\end{equation}
the continuity conditions
\begin{equation}
    \label{cont}
    y_j(\pi) = y_1(\pi), \quad j = \overline{2, m},
\end{equation}
and the nonlocal matching condition (MC)
\begin{equation}
    \label{mcint}
    \sum_{j = 1}^m \left( y_j'(\pi) + \int_0^{\pi} p_j(x) y_j(x) \, dx\right) = 0
\end{equation}
at the central vertex. Here $\lambda$ is the spectral parameter, $\{ h_j \}_{j = 1}^m$ are complex numbers, $p_j \in L_2(0, \pi)$ are complex-valued functions, $j = \overline{1, m}$. 


Spectral theory of differential operators on geometrical graphs has been rapidly developed in recent years (see, e.g., the monographs \cite{BCFK06, BK13, Post12, MP20} and references therein). Such operators have applications in quantum and classical mechanics, nanotechnology, mesoscopic physics, theory of waveguides, etc. The majority of studies in this direction are concerned with local MCs. In particular, the so-called standard MCs of form \eqref{cont}-\eqref{mcint} with $p_j = 0$, $j = \overline{1, m}$, express Kirchoff's law in electrical circuits, balance of tension in elastic string network, etc. (see \cite{BCFK06}). 
However, for modeling some physical processes, nonlocal BCs appear to be more adequate than local ones. 
Differential equations with integral BCs were used in the study of diffusion and heating processes, in the theory of elasticity, biotechnology, and other applications (see \cite{Fel52, Fel54, Krall75, Day82, Schueg87, Yin94, Gor00, SD15}). It is shown in the Appendix of this paper that various models with nonlocal BCs on intervals can be generalized to graph-like structures.

This paper is concerned with the theory of inverse spectral problems, which consist in the recovery of operator coefficients from spectral information. 
We focus on the following inverse problem (IP) for the boundary value problem $L$ of form \eqref{eqv}-\eqref{mcint} with distinct numbers $\{ h_j \}_{j = 1}^m$.

\begin{ip} \label{ip:main}
Suppose that the numbers $\{ h_j \}_{j = 1}^m$ are known a priori.
Given the spectrum $\Lambda$ of the problem $L$, find the functions $\{ p_j \}_{j = 1}^m$.
\end{ip}


Classical results of the inverse spectral theory are presented, e.g., in the monographs \cite{Mar77, Lev84, PT87, FY01}. IPs for differential operators on graphs with local MCs were studied in \cite{Bel04, AK08, YPH11, BF13, Ign15, Yur16, BS17, MT17, XY18} and other papers. Methods for solving inverse spectral problems for differential operators with nonlocal integral BCs on finite intervals were developed by Kravchenko \cite{Krav00}, Yang and Yurko \cite{YY16-jmaa, YY16-amp, YY21}. However, in the mentioned papers, the coefficients of integral BCs are assumed to be known a priori, and the reconstruction of differential expression coefficients is studied. 

The idea of the present paper has been inspired by the study of Kanguzhin et al \cite{Kan21}, which is concerned with recovering coefficients of nonlocal MCs of differential operators on a star-shaped graph from spectra. The similar approach was developed by Kanguzhin \cite{Kan20} for higher-order differential operators on a finite interval. However, the paper \cite{Kan21} has the following principal differences from our study.

\begin{enumerate}
    \item In \cite{Kan21} another type of nonlocal MCs was considered. Those MCs do not generalize the local standard MCs, which are regular and natural for physical applications.
    \item The method of \cite{Kan21} requires the basicity of the eigenfunctions, but it is unclear whether this basicity holds in any cases. In particular, one can easily find an example of non-basicity.
    \item The authors of \cite{Kan21} confine themselves to the identification problem that consists in the reconstruction of the unknown coefficients. The existence of IP solution has not been investigated. In this paper, we not only provide a correct problem statement and develop two constructive methods for solving IP, but also obtain the spectrum characterization for the problem $L$.
\end{enumerate}

Proceed with the formulation of the main results. Denote by $\{ z_k \}_{k = 2}^m$ the roots of the polynomial
\begin{equation} \label{defP}
P(z) = \frac{d}{dz} \prod_{j = 1}^m (z - h_j).
\end{equation}
We will write $h = [h_j]_{j = 1}^m \in \mathbb C^m_*$ if $h_j \in \mathbb C$, $j = \overline{1, m}$, and all the values $\{ h_j \}_{j = 1}^m \cup \{ z_k \}_{k = 2}^m$ are distinct. In particular, $h \in \mathbb C^m_*$ in the special case when $\{ h_j \}_{j = 1}^m$ are real and distinct.

\begin{thm} \label{thm:char}
Suppose that $h \in \mathbb C^m_*$ is fixed.
For a multiset $\Lambda$ to be the spectrum of the problem $L$ of form \eqref{eqv}-\eqref{mcint}, it is necessary and sufficient to be countable and to have a numbering $\Lambda = \{ \lambda_{nk} \}_{n \ge 0, \, k = \overline{1, m}}$ (counting with multiplicities) such that the following asymptotic relations hold
\begin{equation} \label{asymptla}
\arraycolsep=1.4pt\def\arraystretch{2.2}
\left.\begin{array}{l}
    \sqrt{\lambda_{n1}} = n + \dfrac{z_1}{\pi n} + \dfrac{\varkappa_{n1}}{n}, \\
    \sqrt{\lambda_{nk}} = n + \frac{1}{2} + \dfrac{z_k}{\pi (n + \tfrac{1}{2})} + \dfrac{\varkappa_{nk}}{n^2}, \quad k = \overline{2, m},
    \end{array}\quad \right\} \quad n \ge 1,
\end{equation}
where $\{ \varkappa_{nk} \} \in l_2$, $z_1 := \frac{1}{m} \sum\limits_{j = 1}^m h_j$, and $\{ z_k \}_{k = 2}^m$ are the roots of the polynomial~\eqref{defP}.
Moreover, the functions $\{ p_j \}_{j = 1}^m$ of the problem $L$ are uniquely specified by its spectrum~$\Lambda$.
\end{thm}

Note that the sequence $\{ \lambda_{nk} \}_{n \ge 0, \, k = \overline{1, m}}$ can have a finite number of multiple eigenvalues. For large indices $n$, the eigenvalues are simple because of the asymptotics \eqref{asymptla}.

We observe that the numbering of $\{ \lambda_{nk} \}_{n \ge 0, \, k = \overline{1, m}}$ in Theorem~\ref{thm:char} is not uniquely fixed. All the arguments below are valid for any such numbering. We will show that, by using an arbitrary sequence $\{ \la_{nk} \}_{n \ge 0, \, k = \overline{1, m}}$ satisfying \eqref{asymptla}, one can find the functions $\{ p_j \}_{j = 1}^m$ such that $p_j \in L_2(0, \pi)$, $j = \overline{1, m}$, and the eigenvalues of the problem $L$ of form \eqref{eqv}-\eqref{mcint} with these $\{ p_j \}_{j = 1}^m$ and with the initially fixed $\{ h_j \}_{j = 1}^m$ coincide with $\{ \la_{nk} \}_{n \ge 0, \, k = \overline{1, m}}$.

It is worth mentioning that some differential operators with integral BCs are adjoint to functional-differential operators with the so-called \textit{frozen arguments} (see examples in \cite{Lom14, Pol21, BBV19}). In particular, the adjoint problem $L^*$ to the problem \eqref{eqv}-\eqref{mcint} has the form
\begin{gather} \label{eqvz}
    -u_j''(x) + \overline{p_j(x)} u_j(\pi) = \la u_j(x), \quad x \in (0, \pi), \quad j = \overline{1, m}, \\ \label{bcz}
    u_j'(0) - \overline{h_j} u_j(0) = 0, \quad j = \overline{1, m}, \\ \label{mcz}
    u_j(\pi) = u_1(\pi), \quad j = \overline{2, m}, \qquad 
    \sum_{j = 1}^m u_j'(\pi) = 0
\end{gather}
with the frozen argument in \eqref{eqvz} at $x = \pi$, that is, at the central vertex of the graph. The spectra of $L$ and $L^*$ are complex conjugate to each other. Thus, IP~\ref{ip:main} is equivalent to the following problem.

\begin{ip} \label{ip:frozen}
Suppose that the numbers $\{ h_j \}_{j = 1}^m$ are known a priori.
Given the spectrum $\Lambda^*$ of the problem $L^*$, find the functions $\{ p_j \}_{j = 1}^m$.
\end{ip}

Therefore, Theorem~\ref{thm:char} provides the uniqueness of IP~\ref{ip:frozen} solution and the spectrum characterization for the problem $L^*$.


The IP theory for functional-differential operators with frozen arguments on a finite interval has been intensively studied in recent years (see \cite{BBV19, BV19, BK20, HBY20, BH21, WZZW21, Tsai21, Kuz22, Bond22, DH22}). Buterin in the conference proceedings \cite{But21-thesis} has announced some results on IPs for operators with frozen argument on graphs. In addition, we mention the papers \cite{HNA07, Niz09, Niz10, XY19}, concerning the nonlocal self-adjoint operator having both frozen argument and integral BCs. In \cite{Niz12}, analogous operators were considered on graphs. However, the latter operators on graphs cannot be uniquely recovered from the spectrum. The author believes that the results of this paper can be obtained by generalizing the methods of the studies \cite{BBV19, BV19, BK20, HBY20, BH21, WZZW21, Tsai21, Kuz22, Bond22, DH22} concerning operators with frozen argument. Certainly, this will require modification of those methods and technical work. Nevertheless, the development of new methods, related not to the problem $L^*$ with frozen argument but to the problem $L$ with nonlocal MCs, also can be useful for future generalizations and applications. Therefore, this paper aims to provide two new approaches to IPs for differential operators on graphs with nonlocal MCs. Note that our methods can be applied to inverse problems with frozen arguments on finite intervals (see the discussion in Section~\ref{sec:gen}).

The paper is organized as follows. In Section~\ref{sec:asympt}, we derive the eigenvalue asymptotics \eqref{asymptla} and provide several auxiliary lemmas.
In Section~\ref{sec:dif}, we describe our first method of IP solution. This method is based on the Riesz-basis property of the root functions of the problem $L$. It is worth mentioning that the basicity of root functions was studied by Shkalikov \cite{Shkal82} for the higher-order differential operators with integral BCs and by Gomilko and Radzievskii \cite{GR91} for the vectorial functional-differential system $y^{(n)} + Fy = \la y$ with integral BCs.
Anyway, when an IP is investigated, the root functions, strictly speaking, are unknown. In Section~\ref{sec:dif}, we construct a special sequence of vector functions $\{ v_{nk} \}_{n \ge 0, \, k = \overline{1, m}}$ by using arbitrary complex values $\{ \la_{nk} \}_{n \ge 0, \, k = \overline{1, m}}$ satisfying \eqref{asymptla}, not necessarily being eigenvalues of a certain problem $L$. We prove that the sequence $\{ v_{nk} \}_{n \ge 0, \, k = \overline{1, m}}$ is a Riesz basis and then show that $\{ v_{nk} \}_{n \ge 0, \, k = \overline{1, m}}$ are the root functions of some problem $L$. In Section~\ref{sec:easy}, we present another approach, which allows us to recover the function $p_j$ separately on each edge of the graph. The both Sections~\ref{sec:dif} and~\ref{sec:easy} present the proofs of the sufficiency and the uniqueness in Theorem~\ref{thm:char} and constructive algorithms for solving IP~\ref{ip:main}. In Section~\ref{sec:gen}, we consider the case of not necessarily distinct $\{ h_j \}_{j = 1}^m$ and discuss some other generalizations of our results. Note that, if some numbers among $\{ h_j \}_{j = 1}^m$ are equal, the solution of IP~\ref{ip:main} is non-unique. However, the spectrum characterization for the general case is obtained in Theorem~\ref{thm:gen}.
In Appendix, we construct several models of physical processes reduced to the problems $L$ and $L^*$ by the separation of variables.

\section{Preliminaries and asymptotics} \label{sec:asympt}

In this section, prove the necessity in Theorem~\ref{thm:char} and several auxiliary lemmas.
Throughout the paper, we use the following \textbf{notations}.

\begin{itemize}
\item The prime $y'$ denotes the derivative with respect to $x$ and the dot $\dot y$, with respect to $\la$. 
\item Speaking about the uniqueness of the IP solution, along with $L$, we consider another eigenvalue problem $\tilde L$ of the same form \eqref{eqv}-\eqref{mcint} with different coefficients $\tilde p_j \in L_2(0, \pi)$, $j = \overline{1, m}$. The coefficients $\{ h_j \}_{j = 1}^m$ are supposed to be the same for $L$ and $\tilde L$. We agree that, if a symbol $\ga$ denotes an object related to $L$, then the symbol $\tilde \ga$ with tilde denotes the analogous object related to $\tilde L$.
\item The same symbol $C$ is used for various positive constants independent of $x$, $\la$, etc.
\item The same symbol $\{ \varkappa_n \}$ is used for various $l_2$-sequences.
\end{itemize}

Define the functions
\begin{equation} \label{defvv}
\vv_j(x, \la) = \cos(\sqrt{\la} x) + h_j \frac{\sin (\sqrt{\la} x)}{\sqrt{\la}}, \quad j = \overline{1, m}.
\end{equation}
Clearly, for each $j = \overline{1, m}$, the function $\vv_j(x, \la)$ satisfies the corresponding equation \eqref{eqv} and the BC \eqref{bc}. Consequently, the eigenvalues of the problem $L$ coincide with the zeros of the characteristic function
\begin{equation} \label{Delta}
\Delta(\la) = \sum_{j = 1}^m \left( \vv_j'(\pi, \la) + \int_0^{\pi} p_j(x)\vv_j(x, \la)\, dx \right) \prod_{\substack{s = 1 \\ s \ne j}}^m \vv_s(\pi, \la).
\end{equation}

Clearly, the function $\Delta(\la)$ is entire analytic in the $\la$-plane, and the asymptotics of its zeros can be obtained by using the standard approach based on Rouch\'e's theorem. Let us study the zero asymptotics for the following function of a more general form than $\Delta(\la)$: 
\begin{equation} \label{genD}
\mathcal D(\la) := \sum_{j = 1}^m \left( m_j \vv_j'(\pi, \la) + \int_0^{\pi} p_j(x) \vv_j(x, \la) \, dx \right) \prod_{\substack{s = 1 \\ s \ne j}}^m \vv_j(\pi, \la),
\end{equation}
where $\{ h_j \}_{j = 1}^m$ are distinct complex numbers, $m_j$ are some positive integers, $p_j \in L_2(0, \pi)$, $j = \overline{1, m}$. The function of form \eqref{genD} will be used for investigating a generalization of IP~\ref{ip:main} in Section~\ref{sec:gen}.

Denote by $\{ z_k \}_{k = 2}^m$ the roots of the polynomial
\begin{equation} \label{defPg}
\mathcal P(z) := \left( \prod_{j = 1}^m (z - h_j)^{m_j-1} \right)^{-1} \frac{d}{dz} \prod_{j = 1}^m (z - h_j)^{m_j}.
\end{equation}
Suppose that the numbers $\{ h_j \}_{j = 1}^m \cup \{ z_k \}_{k = 2}^m$ are all distinct.

\begin{lem} \label{lem:asympt}
Any function $\mathcal D(\la)$ of form \eqref{genD} has the countable set of zeros 
$\{ \la_{nk} \}_{n \ge 0, \, k = \overline{1, m}}$ (counting with multiplicities) satisfying the asymptotic relations
\begin{equation} \label{asymptg}
\arraycolsep=1.4pt\def\arraystretch{2.2}
\left.\begin{array}{l}
    \sqrt{\lambda_{n1}} = n + \dfrac{z_1}{\pi n} + \dfrac{\varkappa_{n1}}{n}, \\
    \sqrt{\lambda_{nk}} = n + \frac{1}{2} + \dfrac{z_k}{\pi (n + \tfrac{1}{2})} + \dfrac{\varkappa_{nk}}{n^2}, \quad k = \overline{2, m},
    \end{array}\quad \right\} \quad n \ge 1,
\end{equation}
where $\{ \varkappa_{nk} \} \in l_2$, 
\begin{equation} \label{defz1}
z_1 := \left( \sum_{j = 1}^m m_j \right)^{-1} \sum_{j = 1}^m m_j h_j,
\end{equation}
and $\{ z_k \}_{k = 2}^m$ are the roots of the polynomial \eqref{defPg}.
\end{lem}

Clearly, in the special case $m_j = 1$, $j = \overline{1, m}$, Lemma~\ref{lem:asympt} together with \eqref{Delta} imply the necessity in Theorem~\ref{thm:char}. 

\begin{proof}[Proof of Lemma~\ref{lem:asympt}]

By using the standard approach, based on Rouche's theorem (see, e.g., \cite[Theorem~1.1.3]{FY01}), one can easily show that $\mathcal D(\la)$ has a countable set of zeros $\{ \la_{nk} \}_{n \ge 0, \, k = \overline{1, m}}$ with the asymptotics
\begin{align} \label{rla1}
    & \sqrt{\la_{n1}} = n + \frac{z_{n1}}{\pi n}, \\ \nonumber
    & \sqrt{\la_{nk}} = n + \frac{1}{2} + \frac{z_{nk}}{\pi (n + \tfrac{1}{2})}, \quad k = \overline{2, m},
\end{align}
where $z_{nk} = O(1)$, $k = \overline{1, m}$, $n \to \infty$. Our goal is to obtain the more precise asymptotics \eqref{asymptg}.

\textit{Case} $k = 1$. Using \eqref{defvv}, \eqref{genD}, and \eqref{rla1}, we derive the asymptotic formulas
\begin{gather*}
\vv_j(\pi, \la_{n1}) = (-1)^n (1 + O(n^{-2})), \quad 
\vv_j'(\pi, \lambda_{n1}) = (-1)^{n+1} (z_{n1} - h_j + O(n^{-2})), \\
\int_0^{\pi} p_j(x)\vv_j(x, \la_{n1})\, dx = \varkappa_n, \\
\mathcal D(\la_{n1}) = (-1)^{nm+1} \sum_{j = 1}^m m_j (z_{n1} - h_j + \varkappa_n) = 0.
\end{gather*}
Hence $z_{n1} = z_1 + \varkappa_{n1}$, where $z_1$ is defined by \eqref{defz1}, $\{ \varkappa_{n1} \} \in l_2$.
 
\smallskip

\textit{Case} $k = \overline{2, m}$. Consider the mapping
$$
\theta_n(z) = n + \frac{1}{2} + \frac{z}{\pi (n + \tfrac{1}{2})}
$$
of the circle $|z| \le r$, where $r > 0$ is some fixed radius. Calculations show that
\begin{align*}
& \vv_j(\pi, \theta_n^2(z)) = \frac{(-1)^{n + 1}}{n + \tfrac{1}{2}} (z - h_j + O(n^{-2})), \\
& \vv_j'(\pi, \theta_n^2(z)) = (-1)^{n + 1} (n + \tfrac{1}{2}) (1 + O(n^{-2})), \\
& \int_0^{\pi} p_j(x) \vv_j(x, \theta_n^2(z)) \, dx = \varkappa_{nj}(z), \quad \sum_{n} |\varkappa_{nj}(z)|^2 \le C,
\end{align*}
where the estimates are uniform in the circle $|z| \le r$. Substituting these asymptotics into \eqref{genD}, we obtain
$$
\mathcal D(\theta_n^2(z)) = (-1)^{(n+1)m} (n + \tfrac{1}{2})^{-(m-2)} \mathcal P_n(z),
$$
where 
\begin{equation} \label{Pn}
\mathcal P_n(z) = \mathcal P(z) + \frac{\varkappa_n(z)}{n}, \quad \sum_{n} |\varkappa_n(z)|^2 \le C, \quad |z| \le r,
\end{equation}
and $\mathcal P(z)$ is the polynomial defined by \eqref{defPg}.

Recall that the zeros $\{ z_k \}_{k = 2}^m$ of $\mathcal P(z)$ are distinct. Considering \eqref{Pn} on the contours $\{ z \in \mathbb C \colon |z - z_k| = \de \}$, $0 < \de < \min\limits_{j \ne k} |z_j - z_k|$, and applying Rouche's theorem, we conclude that, for sufficiently large $n$, the function $\mathcal P_n(z)$ has the zeros
\begin{equation} \label{znk}
z_{nk} = z_k + \eps_{nk}, \quad \eps_{nk} = o(1), \quad k = \overline{2, m}, \quad n \to \iy.
\end{equation}
Substituting \eqref{znk} into \eqref{Pn} and using the Taylor formula, we obtain
$$
\mathcal P_n(z_{nk}) = \mathcal P(z_k) + \frac{d}{dz} \mathcal P(z_k) \eps_{nk} + O(\eps_{nk}^2) + \frac{\varkappa_n}{n}.
$$
Since the zeros of $\mathcal P(z)$ are simple, this yields $\eps_{nk} = \frac{\varkappa_{nk}}{n}$, $\{ \varkappa_{nk} \} \in l_2$, $k = \overline{2, m}$. Thus, \eqref{asymptg} holds for $k = \overline{2, m}$.
\end{proof}

For the proofs of the main results, we need the following two lemmas related to construction of entire functions by their zeros.

\begin{lem} \label{lem:Delta}
Let $\{ \la_{nk} \}_{n \ge 0, \, k = \overline{1, m}}$ be the eigenvalues of $L$. Then, for the characteristic function $\Delta(\la)$ defined by \eqref{Delta}, the following relation is valid
\begin{equation} \label{prodD}
\Delta(\la) = \prod_{k = 1}^m \Delta_k(\la), 
\end{equation}
where
\begin{equation} \label{Dk}
\Delta_1(\la) = \pi (\la_{01} - \la) \prod\limits_{n = 1}^{\iy} \dfrac{\la_{n1} - \la}{n^2}, \qquad
\Delta_k(\la) = \prod\limits_{n = 0}^{\iy} \dfrac{\la_{nk} - \la}{(n + \tfrac{1}{2})^2}, \quad k = \overline{2, m},
\end{equation}
\end{lem}

Lemma~\ref{lem:Delta} is proved similarly to Theorem~1.1.4 in \cite{FY01}.

\begin{lem} \label{lem:asymptDk}
Let $\{ \la_{nk} \}_{n \ge 0, \, k = \overline{1, m}}$ be arbitrary complex numbers satisfying \eqref{asymptla}. Then the infinite products \eqref{Dk} converge uniformly with respect to $\la$ on compact sets and define entire analytic functions $\Delta_k(\la)$ which can be represented as follows:
\begin{align} \label{aD1}
    & \Delta_1(\la) = -\rho \sin (\rho \pi) + z_1 \cos (\rho \pi) + \int_0^{\pi} \mathscr K(t) \cos (\rho t) \, dt, \\ \label{aDk}
    & \Delta_k(\la) = \cos (\rho \pi) + \dfrac{z_k \sin (\rho \pi)}{\rho} + \dfrac{s_k \cos (\rho \pi)}{\rho^2} + \frac{1}{\rho^2} \int_0^{\pi} \mathscr N_k(t) \cos (\rho t) \, dt, 
    \quad k = \overline{2, m},
\end{align}
where
$$ 
\rho = \sqrt{\la}, \quad
\mathscr K, \mathscr N_k \in L_2(0, \pi), \quad s_k \in \mathbb C, \quad k = \overline{2, m}.
$$
\end{lem}

\begin{proof}
The relation~\eqref{aD1} has been obtained in \cite[Lemma~8]{But19}. The relation~\eqref{aDk} can be similarly derived from \cite[Corollary~2]{Bond18}. Note that all such relations follow from the general Theorem~6 in \cite{But-mz}.
\end{proof}

The following technical lemma is proved by direct calculations.

\begin{lem}
The function $\vv_j(\pi, \la)$ has only simple zeros if and only if
\begin{equation} \label{Zh}
h_j \not\in Z_h, \quad Z_h = \left\{ h = -\frac{\rho}{2 \pi} \cot(\rho/2)\colon \rho \ne 0 \: \textit{satisfies} \: \sin(\rho) = \rho \right\}.
\end{equation}
\end{lem}

\begin{lem} \label{lem:vv}
Let $j \in \{ 1, \ldots, m \}$ be fixed and let the condition~\eqref{Zh} hold. Then the function $\vv_j(\pi, \la)$ has a countable set of simple zeros $\{ \mu_n^{(j)} \}_{n = 0}^{\iy}$ with the asymptotics 
\begin{equation} \label{asymptmu}
\sqrt{\mu_n^{(j)}} = n + \frac{1}{2} + \frac{h_j}{\pi n} + O(n^{-2}), \quad n \to \iy.
\end{equation}
Moreover, the sequence $\{\vv_j(x, \mu_n^{(j)})\}_{n = 0}^{\iy}$ is a Riesz basis in $L_2(0, \pi)$.
\end{lem}

\begin{proof}
Observe that $\{ \mu_n^{(j)} \}_{n = 0}^{\iy}$ and $\{\vv_j(x, \mu_n^{(j)})\}_{n = 0}^{\iy}$ are the eigenvalues and the eigenfunctions, respectively, of the boundary value problem
$$
-y''(x) = \la y(x), \quad y'(0) - h_j y(0) = 0, \quad y(\pi) = 0.
$$
Using Theorem~1.1.3. and Proposition~1.8.6 of \cite{FY01}, we arrive at the assertion of the lemma.
\end{proof}

\section{First method} \label{sec:dif}

In this section, the first method for solving IP~\ref{ip:main} is presented. Let a vector $h = [h_j]_{j = 1}^m \in \mathbb C^m_*$ be fixed. For simplicity, assume that $h_j \not\in Z_h$, $j = \overline{1, m}$, where $Z_h$ is defined in \eqref{Zh}. This restriction is just technical. In the general case, the proofs are valid with minor modifications. Suppose that  
arbitrary complex numbers $\{ \la_{nk} \}_{n \ge 0, \, k = \overline{1, m}}$ satisfying \eqref{asymptla} are given, and we have to find functions $\{ p_j \}_{j = 1}^m$ such that $\{ \la_{nk} \}$ are the eigenvalues of the corresponding problem $L$ of form \eqref{eqv}-\eqref{mcint}.

Let us briefly describe our method.
First, we notice that the root vector functions $\{ y_{nk} \}_{n \ge 0, \, k = \overline{1,m}}$ of the problem $L$ can be constructed by using its eigenvalues and the coefficients $\{ h_j \}_{j = 1}^m$ by the formulas provided in Definition~\ref{df:y} below. If we have some complex numbers $\{ \la_{nk} \}$ with asymptotics \eqref{asymptla} and do not know whether they are eigenvalues of any problem $L$ or not, then we also can construct the functions $\{ y_{nk} \}$ by Definition~\ref{df:y}. Second, we prove the completeness of the sequence $\{ y_{nk} \}_{n \ge 0, \, k = \overline{1,m}}$ (Lemma~\ref{lem:complete}), which leads to the uniqueness in Theorem~\ref{thm:char}. Third, we multiply the vector functions $y_{nk}$ by certain coefficients to obtain new vector functions $v_{nk}$ forming a Riesz basis.
The Riesz-basis property is proved in Lemma~\ref{lem:basis}. 
By using the root function Riesz-basis and the integral MC \eqref{mcint}, it is possible to recover $\{ p_j \}_{j = 1}^m$.
Relying on this idea, we prove the sufficiency in Theorem~\ref{thm:char} obtain a constructive procedure (Algorithm~\ref{alg:dif}) for solving IP~\ref{ip:main}.

Suppose that $\{ \la_{nk} \}_{n \ge 0, \, k = \overline{1, m}}$ are arbitrary complex numbers satisfying \eqref{asymptla}. Construct the sequence of vector functions $\{ y_{nk} \}_{n \ge 0, \, k = \overline{1, m}}$ according to the following definition.

\begin{df} \label{df:y}
For any value $\la_{\diamond}$ from the sequence $\{ \la_{nk} \}_{n \ge 0, \, k = \overline{1, m}}$, define the chain of vector functions $y_{\diamond}^{<0>}$, $y_{\diamond}^{<1>}$, \dots, $y_{\diamond}^{<m_{\diamond} - 1>}$ by the following rules (i), (ii). Here $m_{\diamond}$ is the multiplicity of $\la_{\diamond}$, that is, 
\begin{gather} \nonumber
m_{\diamond} = \# \{ (n, k) \colon \la_{nk} = \la_{\diamond}, \, n \ge 0, \, k = \overline{1, m}\}, \\ \label{multla}
\la_{\diamond} = \la_{n_s k_s}, \quad s = \overline{1, m_{\diamond}}, \qquad (n_s, k_s) < (n_{s + 1}, k_{s + 1}), \quad s = \overline{1, m_{\diamond}-1}.
\end{gather}

(i) If $\vv_j(\pi, \la_{\diamond}) \ne 0$, $j = \overline{1, m}$, then put
$$
y^{<\nu>}_{\diamond j}(x) = \frac{d^{\nu}}{d \la^{\nu}} \left( \frac{\vv_j(x, \la)}{\vv_j(\pi, \la)} \right)\bigg|_{\la = \la_{\diamond}}, \quad j = \overline{1, m}, \quad \nu = \overline{0, m_{\diamond}-1}.
$$

(ii) If $\vv_s(\pi, \la_{\diamond}) = 0$, then put
\begin{gather*}
y^{<0>}_{\diamond s}(x) = \frac{\vv_s(x, \la_{\diamond})}{\dot \vv_s(x, \la_{\diamond})}, \qquad y^{<0>}_{\diamond j}(x) = 0, \quad j = \overline{1, m} \backslash s, \\
y^{<\nu>}_{\diamond s}(x) = \frac{d^{\nu-1}}{d \la^{\nu-1}} \left( \frac{\dot \vv_s(x, \la)}{\dot \vv_s(\pi, \la)} \right)\bigg|_{\la = \la_{\diamond}}, \qquad
y^{<\nu>}_{\diamond j}(x) = \frac{d^{\nu-1}}{d \la^{\nu-1}} \left( \frac{\vv_j(x, \la)}{\vv_j(\pi, \la)} \right)\bigg|_{\la = \la_{\diamond}}, \quad j = \overline{1, m}\backslash s, \\ \nu = \overline{1, m_{\diamond}-1}.
\end{gather*}

In the both cases (i) and (ii), we form the $m$-vector functions $y_{\diamond}^{<\nu>}(x) = [y_{\diamond j}^{<\nu>}(x)]_{j = 1}^m$, $\nu = \overline{0, m_{\diamond}-1}$ and put
$$
y_{n_s k_s}(x) = y_{\diamond}^{<s-1>}(x), \quad s = \overline{1, m_{\diamond}},
$$
according to \eqref{multla}. Thus, the sequence $\{ y_{nk} \}_{n \ge 0, \, k = \overline{1, m}}$ is defined.
\end{df}

Note that, in the case (ii), $\vv_j(\pi, \la_{\diamond}) \ne 0$ for $j \ne s$ and $\dot \vv_s(\pi, \la_{\diamond}) \ne 0$. 
\begin{lem} \label{lem:root}
If $\la_{\diamond}$ is an eigenvalue of the problem $L$ of algebraic multiplicity $m_{\diamond}$ (i.e. the zero of $\Delta(\la)$ of multiplicity $m_{\diamond}$), then $\{ y_{\diamond}^{<\nu>} \}_{\nu = 0}^{m_{\diamond}-1}$ constructed in Definition~\ref{df:y} is the corresponding chain of root functions. This means
\begin{gather} \label{eqr}
-\frac{d^2}{dx^2} y_{\diamond j}^{<0>}(x) = \la_{\diamond} y_{\diamond j}^{<0>}(x), \quad -\frac{d^2}{dx^2} y_{\diamond j}^{<\nu>}(x) = \la_{\diamond} y_{\diamond j}^{<\nu>}(x) + y_{\diamond j}^{<\nu-1>}(x), \quad \nu = \overline{1, m_{\diamond}-1}, \\ \nonumber j = \overline{1, m}, \quad x \in (0, \pi),
\end{gather}
and the conditions \eqref{bc}-\eqref{mcint} are fulfilled for every $y = y^{<\nu>}_{\diamond}$, $\nu = \overline{0, m_{\diamond}-1}$.
\end{lem}

Lemma~\ref{lem:root} is proved by direct calculations. 

Define the Hilbert space of vector functions
$$
\mathcal H = \{ y = [y_j]_{j = 1}^m \colon y_j \in L_2(0, \pi), \, j = \overline{1, m} \}
$$
with the scalar product
$$
(y, u) = \sum_{j = 1}^m \int_0^{\pi} \overline{y_j(x)} u_j(x) \, dx, \quad
y = [y_j]_{j = 1}^m, \: u = [u_j]_{j = 1}^m.
$$
If $\{ y_{nk} \}_{n \ge 0, \, k = \overline{1, m}}$ are the root functions of $L$, then the condition \eqref{mcint} implies
\begin{equation} \label{prody}
    (p, y_{nk}) = -\sum_{j = 1}^m y_{nk,j}'(\pi), \quad n \ge 0, \, k = \overline{1, m},
\end{equation}
where $p = [\overline{p_j}]_{j = 1}^m$, and $y_{nk,j}$ denotes the $j$-th element of the vector function $y_{nk}$.

Now let us consider the sequence $\{ y_{nk} \}_{n \ge 0, \, k = \overline{1, m}}$ constructed via Definition~\ref{df:y} by using some numbers $\{ \la_{nk} \}_{n \ge 0, \, k = \overline{1,m}}$ satisfying \eqref{asymptla} but not necessarily being the eigenvalues of a certain problem $L$. 

\begin{lem} \label{lem:complete}
The sequence $\{ y_{nk} \}_{n \ge 0, \, k = \overline{1, m}}$ is complete in $\mathcal H$.
\end{lem}

\begin{proof}
Let $f = [\overline{f_j}]_{j = 1}^m \in \mathcal H$ be such that $(f, y_{nk}) = 0$ for all $n \ge 0$, $k = \overline{1, m}$. For $\la_{nk} = \la_{\diamond}$, this means
$$
\sum_{j = 1}^m \int_0^{\pi} f_j(x) y_{\diamond j}^{<\nu>}(x) \, dx = 0, \quad \nu = \overline{0, m_{\diamond}-1}.
$$
In the case (i) of Definition~\ref{df:y}, this readily implies that the function
$$
G(\la) = \sum_{j = 1}^m \frac{1}{\vv_j(\pi, \la)}\int_0^{\pi} f_j(x) \vv_j(x, \la) \, dx
$$
has the zero $\la_{\diamond}$ of multiplicity at least $m_{\diamond}$. Consequently, $\la_{\diamond}$ is the zero of the function
\begin{equation} \label{defH}
H(\la) = \sum_{j = 1}^m \int_0^{\pi} f_j(x) \vv_j(x, \la) \, dx \prod_{\substack{s = 1 \\ s \ne j}}^{m} \vv_s(\pi, \la)
\end{equation}
of multiplicity at least $m_{\diamond}$. The latter fact can be proved analogously for the case (ii) of Definition~\ref{df:y}. Thus, the function $H(\la)$ has zeros $\{ \la_{nk} \}_{n \ge 0, \, k = \overline{1, m}}$.

Define $\Delta(\la)$ by \eqref{prodD}-\eqref{Dk}:
$$
\Delta(\la) = \pi (\la_{01} - \la) \prod\limits_{n = 1}^{\iy} \dfrac{\la_{n1} - \la}{n^2} \prod_{k = 2}^m \prod\limits_{n = 0}^{\iy} \dfrac{\la_{nk} - \la}{(n + \tfrac{1}{2})^2}.
$$
Clearly, the function $\dfrac{H(\la)}{\Delta(\la)}$ is entire and
\begin{equation} \label{estH}
|H(\la)| \le C \exp\left(m |\mbox{Im}\, \sqrt{\la}| \pi\right).
\end{equation}
The asymptotics of Lemma~\ref{lem:asymptDk} imply the estimate
\begin{equation} \label{estD}
|\Delta(\rho^2)| \ge C|\rho| \exp\left(m |\mbox{Im}\, \rho| \pi\right), \quad |\rho| > \rho^*,
\end{equation}
in the region
$$
G_{\de} = \{ \rho \in \mathbb C \colon |\rho - \tfrac{n}{2}| \ge \de, \, n \in \mathbb Z \}, \quad \de > 0.
$$
Combining \eqref{estH} and \eqref{estD}, we conclude that 
$$
\frac{H(\la)}{\Delta(\la)} = O\left( \frac{1}{\sqrt \la}\right), \quad |\la| \to \iy.
$$
By Liouville's theorem, $H(\la) \equiv 0$.

Fix $j \in \{ 1, \ldots, m \}$. Denote by $\{ \mu_n^{(j)} \}_{n = 0}^{\iy}$ the zeros of $\vv_j(\pi, \la)$. Using \eqref{defH}, we obtain
$$
\int_0^{\pi} f_j(x) \vv_j(x, \mu_n^{(j)}) \, dx = 0, \quad n \ge 0.
$$
In view of Lemma~\ref{lem:vv}, the sequence $\{ \vv_j(x, \mu_n^{(j)}) \}_{n = 0}^{\iy}$ is complete in $L_2(0, \pi)$, so $f_j(x) = 0$ a.e. on $(0, \pi)$. Hence, $f = 0$ in $\mathcal H$, so $\{ y_{nk} \}_{n \ge 0, \, k = \overline{1, m}}$ is complete.
\end{proof}

Lemma~\ref{lem:complete} together with \eqref{prody} imply the uniqueness in Theorem~\ref{thm:char}.

\begin{proof}[Proof of the uniqueness in Theorem~\ref{thm:char}]
Consider two problems $L$ and $\tilde L$ with the coefficients $\{ p_j \}_{j = 1}^m$ and $\{ \tilde p_j \}_{j = 1}^m$, respectively.
Suppose that the eigenvalues of $L$ and $\tilde L$ coincide (counting with multiplicities), that is, $\la_{nk} = \tilde \la_{nk}$, $n \ge 0$, $k = \overline{1, m}$. Taking Definition~\ref{df:y} and Lemma~\ref{lem:root} into account, we conclude that the root functions of $L$ and $\tilde L$ also coincide: $y_{nk} = \tilde y_{nk}$, $n \ge 0$, $k = \overline{1, m}$. The relation \eqref{prody} and the similar relation for $\tilde L$ imply $(p - \tilde p, y_{nk}) = 0$, $n \ge 0$, $k = \overline{1, m}$, where $p = [\overline{p_j}]_{j = 1}^m$, $\tilde p = [\overline{\tilde p_j}]_{j = 1}^m$. Using Lemma~\ref{lem:complete}, we conclude that $p = \tilde p$ in $\mathcal H$, so $p_j(x) = \tilde p_j(x)$ a.e. on $(0, \pi)$, $j = \overline{1, m}$. Thus, the eigenvalues $\{ \la_{nk} \}_{n \ge 0, \, k = \overline{1, m}}$ uniquely specify the coefficients $\{ p_j \}_{j = 1}^m$.
\end{proof}

Define
\begin{equation} \label{defv}
v_{n1} = (-1)^n y_{n1}, \quad v_{nk} = \frac{(-1)^n}{n + \tfrac{1}{2}} y_{nk}, \quad k = \overline{2, m},
\end{equation}

\begin{lem} \label{lem:basis}
The sequence $\{ v_{nk} \}_{n \ge 0, \, k = \overline{1,m}}$ is a Riesz basis in $\mathcal H$.
\end{lem}

Necessary information about Riesz bases can be found in \cite[Section~1.8.5]{FY01} and in \cite{Christ03}. In particular, the proof of Lemma~\ref{lem:basis} is based on the following proposition (Proposition~1.8.5 from \cite{FY01}).

\begin{prop} \label{prop:Riesz}
Suppose that a sequence $\{ f_n \}_{n \ge 1}$ is complete in a Hilbert space $B$ and quadratically close to some Riesz basis $\{ g_n \}_{n \ge 1}$ in $B$, that is,
$$
\sum_{n = 1}^{\iy} \| f_n - g_n \|^2_B < \infty. 
$$
Then $\{ f_n \}_{n = 1}^{\iy}$ is also a Riesz basis in $B$.
\end{prop}

\begin{proof}[Proof of Lemma~\ref{lem:basis}]
By virtue of Lemma~\ref{lem:complete} and \eqref{defv}, the sequence $\{ v_{nk} \}_{n \ge 0, \, k = \overline{1, m}}$ is complete in $\mathcal H$. 
In view of the asymptotics \eqref{asymptla} and Definition~\ref{df:y}, for sufficiently large $n$, the values $\{ \la_{nk} \}$ are simple, so
\begin{equation} \label{vvv}
v_{n1,j}(x) = (-1)^n \frac{\vv_j(x, \la_{n1})}{\vv_j(\pi, \la_{n1})}, \quad 
v_{nk,j}(x) = \frac{(-1)^n \vv_j(x, \la_{nk})}{(n + \tfrac{1}{2}) \vv_j(\pi, \la_{nk})}, \quad k = \overline{2, m}, \quad j = \overline{1, m}.
\end{equation}
Therefore, using \eqref{asymptla} and \eqref{defvv}, one can easily show that the sequence $\{ v_{nk} \}_{n \ge 0, \, k = \overline{1, m}}$ is
quadratically close to the sequence $\{ v_{nk}^0 \}_{n \ge 0, \, k = \overline{1, m}}$:
$$
v_{n1}^0 = [\cos(nx)]_{j = 1}^m, \quad v_{nk}^0 = \left[(h_j - z_k)^{-1} \cos((n + \tfrac{1}{2})x)\right]_{j = 1}^m, \quad k = \overline{2, m}.
$$

Let us show that $\{ v_{nk}^0 \}_{n \ge 0, \, k = \overline{1, m}}$ is a Riesz basis in $\mathcal H$.
It is easy to check that
\begin{gather*}
0 < c_1 \le \| v_{nk}^0 \| \le c_2, \quad n \ge 0, \, k = \overline{1, m}, \\
(v^0_{nk}, v^0_{ls}) = 0, \quad n \ne l, \qquad (v^0_{n1}, v^0_{nk}) = 0, \quad k = \overline{2, m}, \\
(v^0_{nk}, v^0_{ns}) = \frac{2}{\pi} a_{ks}, \quad a_{ks} = \sum_{j = 1}^m \frac{1}{\overline{(h_j - z_k)} (h_j - z_s)}, \quad k, s = \overline{2, m}.
\end{gather*}
Clearly, $A = [a_{ks}]_{k, s = 2}^m$ is the Gram matrix of the vectors $f_k = [(h_j - z_k)]_{j = 1}^m$, $k = \overline{2, m}$. Since the numbers $\{ h_j \}_{j = 1}^m \cup \{ z_k \}_{k = 2}^m$ are all distinct, then it can be shown that the vectors $\{ f_k \}_{k = 2}^m$ are linearly independent. Hence, $\det A \ne 0$. Consequently, $\{ v_{nk}^0 \}_{n \ge 0, \, k = \overline{1, m}}$ is a Riesz basis.
This together with Proposition~\ref{prop:Riesz} yield the claim.
\end{proof}

If the vector functions $\{ v_{nk} \}_{n \ge 0, \, k = \overline{1, m}}$ are constructed by using the eigenvalues $\{ \la_{nk} \}_{n \ge  0, \, k = \overline{1, m}}$ of a certain eigenvalue problem $L$, then,
using \eqref{prody} and \eqref{defv}, we obtain
\begin{equation} \label{prodv}
    (p, v_{nk}) = \eta_{nk}, \quad n \ge 0, \, k = \overline{1, m},
\end{equation}
where
\begin{equation} \label{defeta}
\eta_{nk} := -\sum_{j = 1}^m v_{nk, j}'(\pi), \quad n \ge 0, \quad k = \overline{1, m}.
\end{equation}

\begin{lem} \label{lem:eta}
The sequence $\{ \eta_{nk} \}$ defined by \eqref{defeta} belongs to $l_2$.
\end{lem}

\begin{proof}
It follows from \eqref{vvv}, that
\begin{equation} \label{vt}
v'_{n1,j}(\pi) = (-1)^n \frac{\vv_j'(\pi, \la_{n1})}{\vv_j(\pi, \la_{n1})}, \quad 
v'_{nk,j}(\pi) = \frac{(-1)^n \vv_j'(\pi, \la_{nk})}{(n + \tfrac{1}{2}) \vv_j(\pi, \la_{nk})}, \quad k = \overline{2, m}, \quad j = \overline{1, m},
\end{equation}
for sufficiently large $n$. Using \eqref{defvv} and \eqref{asymptla}, we obtain
\begin{align*}
    & \vv_j'(\pi, \la_{n1}) = (-1)^n (h_j - z_1 + \varkappa_n), \quad \vv_j(\pi, \la_{n1}) = (-1)^n + O(n^{-1}), \\
    & \vv_j'(\pi, \la_{nk}) = (-1)^{n + 1} (n + \tfrac{1}{2}) + O(n^{-1}), \quad
    \vv_j(\pi, \la_{nk}) = \frac{(-1)^n}{(n + \tfrac{1}{2})} (h_j - z_k + \tfrac{\varkappa_n}{n}), \quad k = \overline{2, m}.
\end{align*}
Substituting these asymptotics into \eqref{vt}, \eqref{defeta}, we get
\begin{align*}
& \eta_{n1} = (-1)^{n+1} \left( \sum_{j = 1}^m h_j - m z_1 \right) + \varkappa_n, \\
& \eta_{nk} = (-1)^{n + 1} (n + \tfrac{1}{2}) \sum_{j = 1}^m \left( \frac{1}{z_k - h_j} + \frac{\varkappa_n}{n}\right), \quad k = \overline{2, m}.
\end{align*}
Recall that $z_1 = \frac{1}{m} \sum\limits_{j = 1}^m h_j$ and $\{ z_k \}_{k = 2}^m$ are the roots of the polynomial $P(z)$ defined by \eqref{defP}, so
$$
\sum_{j = 1}^m \frac{1}{z_k - h_j} = P(z_k) \prod_{j =1}^m \frac{1}{z_k - h_j} = 0.
$$
Hence $\{ \eta_{nk} \} \in l_2$.
\end{proof}

\begin{proof}[Proof of the sufficiency in Theorem~\ref{thm:char}]
Suppose that $h \in \mathbb C^m_*$ and $\{ \la_{nk} \}_{n \ge 0, \, k = \overline{1, m}}$ are arbitrary complex numbers satisfying \eqref{asymptla}. Construct vector functions $\{ y_{nk} \}_{n \ge 0, \, k = \overline{1, m}}$ by Definition~\ref{df:y} and then
$\{ v_{nk} \}_{n \ge 0, \, k = \overline{1, m}}$ by \eqref{defv}. 
By Lemma~\ref{lem:basis}, $\{ v_{nk} \}_{n \ge 0, \, k = \overline{1, m}}$ is a Riesz basis in $\mathcal H$.
Define $\eta_{nk}$ by \eqref{defeta}.
By virtue of Lemma~\ref{lem:eta}, $\{ \eta_{nk} \} \in l_2$.
Hence, there exists a unique vector function $p \in \mathcal H$, $p = [\overline{p_j}]_{j = 1}^m$, satisfying \eqref{prodv}. Let $L$ be the problem \eqref{eqv}-\eqref{mcint} with these functions $\{ p_j \}_{j = 1}^m$ in the MC \eqref{mcint}. It remains to prove that $\{ \la_{nk} \}_{n \ge 0, \, k = \overline{1, m}}$ are the eigenvalues of $L$.

The relations \eqref{prodv} and \eqref{defv} imply \eqref{prody}, so each $y = y_{nk}$ satisfies \eqref{mcint}. Using Definition~\ref{df:y}, it is easy to check that each $y_{nk}$ satisfies \eqref{bc}, \eqref{cont} and that \eqref{eqr} is valid. Combining the above arguments, we conclude that $\{ \la_{nk} \}_{n \ge 0, \, k = \overline{1, m}}$ are the eigenvalues and
$\{ y_{nk} \}_{n \ge 0, \, k = \overline{1, m}}$ are the corresponding root functions of the problem $L$.
\end{proof}

The described proof is constructive and yields the following algorithm for solving IP~\ref{ip:main}.

\begin{alg} \label{alg:dif}
Suppose that $h = [h_j]_{j = 1}^m \in \mathbb C^m_*$ and complex numbers $\{ \la_{nk} \}_{n \ge 0, \, k = \overline{1, m}}$ satisfying \eqref{asymptla} are given. We have to find $\{ p_j \}_{j = 1}^m$.
\begin{enumerate}
    \item Construct $\{ y_{nk} \}_{n \ge 0, \, k = \overline{1, m}}$ by Definition~\ref{df:y}.
    \item Find $\{ v_{nk} \}_{n \ge 0, \, k = \overline{1, m}}$ by \eqref{defv} 
    \item Construct the biorthonormal basis $\{ v_{nk}^* \}_{n \ge 0, \, k = \overline{1, m}}$ such that
    $$
        (v_{nk}, v_{ls}^*) = \begin{cases}
                                1, \quad (n,k) = (l, s), \\
                                0, \quad \text{otherwise}.
                            \end{cases}
    $$
    \item Find $\{ \eta_{nk} \}_{n \ge 0, \, k = \overline{1, m}}$ by \eqref{defeta}.
    \item Find $p \in \mathcal H$, $p = [\overline{p_j}]_{j= 1}^m$ satisfying \eqref{prodv} by
    $$
        p = \sum_{n = 0}^{\iy} \sum_{k = 1}^m \overline{\eta_{nk}} v^*_{nk}.
    $$
\end{enumerate}
\end{alg}

\section{Second method} \label{sec:easy}

Theorem~\ref{thm:char} can be proved easier by another method. Define the functions
\begin{align} \label{Delta0}
    & \Delta_0(\la) := \sum_{j = 1}^m \vv_j'(\pi, \la) \prod_{\substack{s = 1\\ s \ne j}}^m \vv_s(\pi, \la), \\ \label{hDelta}
    & \hat \Delta(\la) := \Delta(\la) - \Delta_0(\la) = \sum_{j = 1}^m \int_0^{\pi}p_j(x)\vv_j(x, \la)\, dx \prod_{\substack{s = 1\\ s \ne j}}^m \vv_s(\pi, \la),
\end{align}
where $\Delta(\la)$ was defined in \eqref{Delta}.

For each $j = \overline{1, m}$, consider the zeros $\{ \mu_n^{(j)} \}_{n = 0}^{\iy}$ of the function $\vv_j(\pi, \la)$ defined by \eqref{defvv}. Under our restrictions, $\vv_s(\pi, \mu_n^{(j)}) \ne 0$ for $s \ne j$. Hence
\begin{gather} \nonumber
\hat \Delta(\mu_n^{(j)}) = \int_0^{\pi} p_j(x) \vv_j(x, \mu_n^{(j)}) \, dx \prod_{\substack{s = 1\\ s \ne j}}^m \vv_s(\pi, \mu_n^{(j)}), \\ \label{prodvv}
\int_0^{\pi} p_j(x) \vv_j(x, \mu_n^{(j)}) \,dx = \chi_{n}^{(j)}, \quad j = \overline{1, m}, \quad n \ge 0, \\ \label{defchi}
\chi_{n}^{(j)} := \hat \Delta(\mu_n^{(j)}) \prod_{\substack{s = 1\\ s \ne j}}^m \vv_s^{-1}(\pi, \mu_n^{(j)}).
\end{gather}

Assume \eqref{Zh} for $j = \overline{1, m}$.
By Lemma~\ref{lem:vv},
for each fixed $j$, the sequence $\{ \vv_j(x, \mu_n^{(j)}) \}_{n = 0}^{\iy}$ is a Riesz basis in $L_2(0, \pi)$. Therefore, equations \eqref{prodvv} can be used to determine $p_j$. Thus, we arrive at the following algorithm for the IP solution.

\begin{alg} \label{alg:easy}
Suppose that $h = [h_j]_{j = 1}^m \in \mathbb C^m_*$ and complex numbers $\{ \la_{nk} \}_{n \ge 0, \, k = \overline{1, m}}$ satisfying \eqref{asymptla} are given. We have to find $\{ p_j \}_{j = 1}^m$.

\begin{enumerate}
    \item Construct $\Delta(\la)$ by \eqref{prodD}-\eqref{Dk}:
\begin{equation} \label{smD}
\Delta(\la) = \pi (\la_{01} - \la) \prod\limits_{n = 1}^{\iy} \dfrac{\la_{n1} - \la}{n^2} \prod_{k = 2}^m \prod\limits_{n = 0}^{\iy} \dfrac{\la_{nk} - \la}{(n + \tfrac{1}{2})^2}.
\end{equation}
    \item Find $\Delta_0(\la)$ by \eqref{Delta0} and $\hat \Delta(\la) = \Delta(\la) - \Delta_0(\la)$.
    \item For each $j = \overline{1, m}$, implement the following steps 4--7.
    \item Find the zeros $\{ \mu_n^{(j)} \}_{n = 0}^{\iy}$ of $\vv_j(\pi, \la)$.
    \item Find the sequence $\{ w_{n}^{(j)} \}_{n = 0}^{\iy}$ biorthonormal to $\{ \vv_j(x, \mu_n^{(j)}) \}_{n = 0}^{\iy}$, that is,
    $$
        \int_0^{\pi} \overline{w_n^{(j)}(x)} \vv_j(x, \mu_k^{(j)}) \, dx = \begin{cases}
            1, \quad n = k, \\
            0, \quad \text{otherwise}.
        \end{cases}
    $$
    \item Find $\{ \chi_n^{(j)} \}_{n = 0}^{\iy}$ by \eqref{defchi}.
    \item Find $p_j \in L_2(0, \pi)$ satisfying \eqref{prodvv} by the formula
    \begin{equation} \label{pj}
        p_j(x) = \sum_{n = 0}^{\iy} \chi_n^{(j)} w_n^{(j)}(x).
    \end{equation}
\end{enumerate}
\end{alg}

In contrast to Algorithm~\ref{alg:dif}, Algorithm~\ref{alg:easy} finds the functions $\{ p_j \}_{j = 1}^m$ separately for each~$j$.

Below, we provide the proof of sufficiency and uniqueness in Theorem~\ref{thm:char} based on Algorithm~\ref{alg:easy}. First, we need the following technical lemma.

\begin{lem} \label{lem:chi}
The sequence $\{ \chi_n^{(j)} \}_{n = 0}^{\iy}$ defined by \eqref{defchi} belongs to $l_2$.
\end{lem}

\begin{proof}
Fix $j \in \{ 1, \dots, m \}$.
Using \eqref{defvv} and \eqref{asymptmu}, we obtain
\begin{equation} \label{sm1}
\vv_s(\pi, \mu_n^{(j)}) = (-1)^n (n + \tfrac{1}{2})^{-1} (h_s - h_j + O(n^{-2})), \quad s \ne j, \quad n \to \iy.
\end{equation}
The asymptotics of Lemmas~\ref{lem:asymptDk} and~\ref{lem:vv} together imply
\begin{align*}
& \Delta_1(\mu_n^{(j)}) = (-1)^{n + 1} (n + \tfrac{1}{2}) \left( 1 + \tfrac{\varkappa_n}{n}\right), \\
& \Delta_k(\mu_n^{(j)}) = (-1)^n (n + \tfrac{1}{2})^{-1} \left(z_k - h_j + \tfrac{\varkappa_n}{n}\right), \quad k = \overline{2, m}.
\end{align*}
Substituting the latter asymptotics into \eqref{prodD}, we derive
$$
\Delta(\mu_n^{(j)}) = (-1)^{nm + 1} (n + \tfrac{1}{2})^{-(m-2)} \left( \prod_{k = 2}^m (z_k - h_j) + \frac{\varkappa_n}{n}\right).
$$
The similar asymptotic formula is valid for $\Delta_0(\la)$ defined by \eqref{Delta0}. Hence
\begin{equation} \label{sm2}
\hat \Delta(\mu_n^{(j)}) = n^{-(m - 1)} \varkappa_n.
\end{equation}
Substituting \eqref{sm1} and \eqref{sm2} into \eqref{defchi}, we conclude that $\{ \chi_n^{(j)} \} \in l_2$.
\end{proof}

\begin{proof}[Proof of Theorem~\ref{thm:char}]
\textit{Step 1.}
Suppose that $h \in \mathbb C^m_*$ and $\{ \la_{nk} \}_{n \ge 0, \, k = \overline{1, m}}$ are arbitrary complex numbers satisfying \eqref{asymptla}. Let us show that all the steps of Algorithm~\ref{alg:easy} are correct and as a result determine functions $p_j \in L_2(0, \pi)$, $j = \overline{1, m}$.

By virtue of Lemma~\ref{lem:asymptDk}, the infinite products in \eqref{smD} converge absolutely and uniformly on compact sets and determine the entire function $\Delta(\la)$. 
Fix $j \in \{ 1, \dots, m\}$.
By Lemma~\ref{lem:vv}, the sequence $\{ \vv_j(x, \mu_n^{(j)}) \}_{n = 0}^{\iy}$ is a Riesz basis in $L_2(0, \pi)$. Therefore, there exists the biorthonormal sequence $\{ w_n^{(j)} \}_{n = 0}^{\iy}$ also being a Riesz basis. By Lemma~\ref{lem:chi}, the sequence $\{ \chi_n^{(j)} \}_{n = 0}^{\iy}$ belongs to $l_2$. Thus, the series \eqref{pj} converges in $L_2(0, \pi)$ and defines the function $p_j \in L_2(0, \pi)$ for each $j = \overline{1, m}$.

\smallskip

\textit{Step 2.}
Consider the problem $L$ of form \eqref{eqv}-\eqref{mcint} with the constructed functions $p_j \in L_2(0, \pi)$, $j = \overline{1, m}$. Further, we have to prove that the eigenvalues of this problem $L$ coincide with the initially given $\{ \la_{nk} \}_{n \ge 0, \, k = \overline{1, m}}$. Define the characteristic function $\Delta^{\bullet}(\la)$ of the problem $L$ by \eqref{Delta} and put
$\hat \Delta^{\bullet}(\la) = \Delta^{\bullet}(\la) - \Delta_0(\la)$. By construction, one can easily show that 
\begin{equation} \label{sm3}
\hat \Delta(\mu_n^{(j)}) = \hat \Delta^{\bullet}(\mu_n^{(j)}), \quad j = \overline{1, m}, \quad n \ge 0.
\end{equation}

Using the relation
$$
    \hat \Delta^{\bullet}(\la) = \sum_{j = 1}^m \int_0^{\pi}p_j(x)\vv_j(x, \la)\, dx \prod_{\substack{s = 1\\ s \ne j}}^m \vv_s(\pi, \la),
$$
we get
\begin{equation} \label{sm4}
\hat \Delta^{\bullet}(\la) = o(\exp(m |\mbox{Im} \sqrt{\la}|\pi)), \quad |\la| \to \iy.
\end{equation}
The same asymptotics can be obtained for $\hat \Delta(\la)$ by using \eqref{prodD} and Lemma~\ref{lem:asymptDk}. Consider the function
$$
F(\la) := (\hat \Delta(\la) - \hat \Delta^{\bullet}(\la)) \prod_{j = 1}^m \vv_j^{-1}(\pi, \la).
$$
In view of \eqref{sm3}, $F(\la)$ is entire in $\la$. Using \eqref{defvv} and \eqref{sm4}, we show that $F(\la) = o(1)$ as $|\la| \to \iy$. By Liouville's theorem, $F(\la) \equiv 0$. Consequently, $\Delta(\la) \equiv \Delta^{\bullet}(\la)$, so the initially given values $\{ \la_{nk} \}_{n \ge 0, \, k = \overline{1, m}}$ coincide with the zeros of the characteristic function $\Delta^{\bullet}(\la)$ of the problem $L$. This yields the claim of Step~2.

\smallskip

\textit{Step 3.} It remains to prove the uniqueness of the IP solution. Consider two problems $L$ and $\tilde L$ with coefficients $\{ p_j \}_{j = 1}^m$ and $\{ \tilde p_j \}_{j = 1}^m$, respectively. Suppose that their eigenvalues coincide, that is, $\la_{nk} = \tilde \la_{nk}$, $n \ge 0$, $k = \overline{1, m}$. Then, $\Delta(\la) \equiv \tilde \Delta(\la)$ and so $\hat \Delta(\mu_n^{(j)}) = \tilde{\hat \Delta}(\mu_n^{(j)})$, $j = \overline{1, m}$, $n \ge 0$. It follows from \eqref{defchi} that $\chi_n^{(j)} = \tilde \chi_n^{(j)}$ for $j = \overline{1, m}$, $n \ge 0$. Taking the completeness of $\{ \vv_j(x, \mu_n^{(j)}) \}_{n = 0}^{\iy}$ in $L_2(0, \pi)$ into account, we conclude from \eqref{prodvv} that $p_j = \tilde p_j$ in $L_2(0, \pi)$. This completes the proof. 
\end{proof}

\section{Generalizations} \label{sec:gen}

In this section, some possible generalizations (I-V) and the connection of our results with the problems of \cite{BBV19, BV19, BK20, WZZW21, DH22} (VI-VII) are discussed.

\smallskip

\textbf{I.} First, notice that the condition $h_j \ne h_s$ for $j \ne s$ is crucial for the uniqueness of the IP solution. Indeed, suppose that, on the contrary, $h_j = h_s$ for some $j \ne s$. Then every term of the characteristic function $\Delta(\la)$ defined by \eqref{Delta} contains the multiplier $\vv_s(\pi, \la)$. Therefore, the spectrum $\Lambda$ contains the subsequence $\{ \mu_n^{(s)} \}_{n = 0}^{\iy}$ of the zeros of $\vv_s(\pi, \la)$, and this subsequence carries no information on the functions $\{ p_j \}_{j = 1}^m$.

Let us rigorously formulate the result for not necessarily distinct coefficients $\{ h_j \}_{j = 1}^m$. Without loss of generality, we assume that equal values among $\{ h_j \}_{j = 1}^m$ are consecutive:
$$
h_j = h_{j + 1} = \ldots = h_{j + m_j - 1}, \quad j \in S,
$$
where
$$
S := \{ 1 \} \cup \{j = \overline{2, m} \colon h_{j-1} \ne h_j \}, \quad
m_j := \# \{ s = \overline{1, m} \colon h_s = h_j \}, \quad j \in S.
$$

Then the characteristic function \eqref{Delta} can be represented in the form
\begin{equation} \label{Deltag}
\Delta(\la) = \left( \prod_{j \in S} \vv_j^{m_j-1}(\pi, \la) \right) \sum_{s \in S} \left( m_s \vv_s'(\pi, \la) + \sum_{l = 0}^{m_s - 1} \int_0^{\pi} p_{s + l}(x) \vv_s(x, \la) \, dx\right) \prod_{\substack{k \in S \\ k \ne s}} \vv_k(\pi, \la).
\end{equation}

Denote by $\{ z_k \}_{k \in S \setminus \{1\}}$ the roots of the polynomial
\begin{equation} \label{defP2}
P(z) := \left( \prod_{j \in S} (z - h_j)^{m_j-1} \right)^{-1} \frac{d}{dz} \prod_{j \in S} (z - h_j)^{m_j}.
\end{equation}
We will write $h = [h_j]_{j = 1}^m \in \mathbb C^m_{\bullet}$ if all the numbers $\{ h_j \}_{j \in S} \cup \{ z_k \}_{k \in S \setminus \{1\}}$ are distinct. In particular, this holds if $\{ h_j \}_{j = 1}^m$ are real.

Lemma~\ref{lem:asympt} applied to the function \eqref{Deltag} implies the necessity in the following theorem, generalizing Theorem~\ref{thm:char}.

\begin{thm} \label{thm:gen}
Let $h \in \mathbb C^m_{\bullet}$ be fixed.
For a multiset $\Lambda$ to be the spectrum of the problem $L$ of form \eqref{eqv}-\eqref{mcint}, it is necessary and sufficient to have the form 
$$
\Lambda = \{ \mu_n^{(j)} \}_{n \ge 0, \, j = \overline{1, m} \setminus S} \cup \{ \la_{nk} \}_{n \ge 0, \, k \in S},
$$
where $\{ \mu_n^{(j)} \}_{n = 0}^{\iy}$ are the zeros of $\vv_j(\pi, \la)$, $j = \overline{1, m}$,
\begin{align*}
    & \sqrt{\la_{n1}} = n + \frac{z_1}{\pi n} + \frac{\varkappa_{n1}}{n}, \\
    & \sqrt{\la_{nk}} = n + \frac{1}{2} + \frac{z_k}{\pi (n + \tfrac{1}{2})} + \frac{\varkappa_{nk}}{n^2}, \quad k \in S \setminus \{ 1 \}, 
\end{align*}
$\{ \varkappa_{nk} \} \in l_2$, $z_1 := \frac{1}{m} \sum\limits_{j \in S} m_j h_j$,
and $\{ z_k \}_{k \in S \setminus \{1\}}$ are the roots of the polynomial $P(z)$ defined by \eqref{defP2}.
Moreover, the spectrum $\Lambda$ uniquely specifies the sums
\begin{equation} \label{sump}
\sum_{l = 0}^{m_s-1} p_{s + l}(x), \quad s \in S.
\end{equation}
\end{thm}

Obviously, the part $\{ \mu_n^{(j)} \}_{n \ge 0, \, j = \overline{1, m} \setminus S}$ of the spectrum $\Lambda$ carries no information on the functions $\{ p_j \}_{j = 1}^m$.
The sums \eqref{sump} can be recovered from the part $\{ \la_{nk} \}_{n \ge 0, \, k = \overline{1, m}}$ of the spectrum $\Lambda$ by using Algorithm~\ref{alg:easy} with necessary modifications. Thus, the sufficiency and the uniqueness in Theorem~\ref{thm:gen} can be proved by repeating the arguments of Section~\ref{sec:easy}.

\smallskip

\textbf{II.} The results of this paper can be generalized to the case of the Sturm-Liouville system
\begin{equation} \label{StL}
-y_j'' + q_j(x) y_j = \la y_j, \quad x \in (0, \pi), \quad j = \overline{1, m},
\end{equation}
where $q_j$ are real-valued functions from $L_2(0, \pi)$, $j = \overline{1, m}$, with the BCs \eqref{bc} and the MCs \eqref{cont}-\eqref{mcint}. The functions $\{ q_j \}_{j = 1}^m$ should be known a priori.

In this case, the following two difficulties arise:
\begin{enumerate}
    \item For each fixed $j = \overline{1, m}$, denote by $\vv_j(x, \la)$ the solution of equation \eqref{StL} satisfying the initial conditions $\vv_j(0, \la) = 1$, $\vv_j'(0, \la) = h_j$.
    Some pairs of the functions $\{ \vv_j(\pi, \la) \}_{j = 1}^m$ may have common zeros, which are the eigenvalues of the problem \eqref{StL}, \eqref{bc}-\eqref{mcint} and carry no information about the functions $\{ p_j \}_{j = 1}^m$.
    \item Note that the asymptotic formula for $k = \overline{2, m}$ in \eqref{asymptla} contains the remainder $\frac{\varkappa_{nk}}{n^2}$. This is important for the investigation of IP~\ref{ip:main}. The estimate $\frac{\varkappa_{nk}}{n}$ of the remainder term is insufficient for the IP solution. For the Sturm-Liouville problem \eqref{StL}, \eqref{bc}-\eqref{mcint} even with $p_j = 0$, $j = \overline{1, m}$, the known methods (see \cite{MP15, Bond19}) imply the eigenvalue asymptotics \eqref{asymptla} with some reals $\{ z_k \}_{k = 1}^m$ and with the remainder $\frac{\varkappa_{nk}}{n}$ instead of $\frac{\varkappa_{nk}}{n^2}$ for $k = \overline{2, m}$.
\end{enumerate}

{\bf III.} The case of different edge lengths $\{ l_j \}_{j = 1}^m$ of the star-shaped graph also can be considered. If the functions $\{ \vv_j(l_j, \la) \}_{j = 1}^m$ do not have common zeros, then the method of Section~\ref{sec:easy}, obviously, can be applied by necessity.

\smallskip

{\bf IV.} The methods of this paper can be applied to investigation of IPs for differential operators on graphs with other types of nonlocal MCs. For example, the condition \eqref{mcint} can be replaced by the following one:
$$
\sum_{j = 1}^m \left( y_j'(\pi) + \int_0^{\pi} p_j(x) y_j'(x) \, dx \right) = 0, \quad p_j \in L_2(0, \pi).
$$

{\bf V.} For a more general graph structure, it is natural to consider eigenvalue problems with several frozen arguments at the graph vertices. The equations on the graph edges take the form similar to \eqref{eqvz} with $u_j(\pi)$ replaced by a linear combination of the function values at the vertices. From the physical point of view, such model corresponds to a structure with several sensors located at the graph vertices (see the Appendix).

\smallskip

{\bf VI.} Using the methods of this paper, one can obtain the results of the studies \cite{BBV19, BV19, BK20, WZZW21} concerning the inverse problems for the functional-differential operators with frozen argument on a finite interval in another way. Let us describe the main idea. Consider the boundary value problem
\begin{gather} \label{froz1}
    -u''(x) + q(x) u(a) = \la u(x), \quad x \in (0, 1), \\ \label{froz2}
    u^{(\alpha)}(0) = u^{(\beta)}(1) = 0,
\end{gather}
where $a \in (0, 1)$, $q \in L_2(0, 1)$, $\alpha, \beta \in \{ 0, 1 \}$. The adjoint problem has the form
\begin{gather*}
    -y''(x) = \la y(x), \quad x \in (0, 1), \quad y^{(\alpha)}(0) = y^{(\beta)}(1) = 0, \\
    y'(a + 0) - y'(a - 0) = \int_0^1 \overline{q(x)} y(x) \, dx.
\end{gather*}
The latter problem is equivalent to the following boundary value problem on the two-edge graph with integral MC:
\begin{gather} \label{Lap2}
    -y_j''(x_j) = \la y_j(x_j), \quad x_j \in (0, l_j), \quad j = 1, 2, \\ \label{fbc2}
    y_j^{(\alpha_j)}(0) = 0, \quad j = 1, 2, \\ \label{fmc2}
    y_1(l_1) = y_2(l_2), \quad \sum_{j = 1}^2 \left( y_j'(l_j) + \int_0^{l_j} p_j(x_j) y_j(x_j) \,dx_j \right) = 0,
\end{gather}
where $l_1 = a$, $l_2 = 1-a$, $y_1(x_1) = y(x_1)$, $y_2(x_2) = y(1-x_2)$, $\alpha_1 = \al$, $\al_2 = \beta$, $p_1(x_1) = \overline{q(x_1)}$, $p_2(x_2) = \overline{q(1-x_2)}$, $x_j \in [0, l_j]$. Hence, the recovery of the potential $q$ from the spectrum of the problem \eqref{froz1}-\eqref{froz2} is reduced to the following IP.

\begin{ip} \label{ip:But}
Given the spectrum $\{ \la_n \}_{n = 1}^{\iy}$ of the problem \eqref{Lap2}-\eqref{fmc2}, find $p_1$ and $p_2$.
\end{ip}

For $j = 1, 2$, denote by $\vv_j(x_j, \la)$ the solution of equation \eqref{Lap2} satisfying the initial conditions $\vv_j^{(\alpha_j)}(0, \la) = 0$, $\vv_j^{(1-\alpha_j)}(0, \la) = 0$, that is,
$$
\vv_j(x, \la) = \begin{cases}
                    (\sqrt{\la})^{-1} \sin(\sqrt \la x), \quad \al_j = 0, \\
                    \cos(\sqrt \la x), \quad \al_j = 1.
                \end{cases}
$$
The eigenvalues of the problem \eqref{Lap2}-\eqref{fmc2} coincide with the zeros of the characteristic function
$$
\Delta(\la) = \sum_{j = 1}^2 \left( \vv_j'(l_j, \la) + \int_0^{l_j} p_j(x_j) \vv_j(x_j, \la) \, d x_j \right) \vv_{3-j}(l_{3-j}, \la).
$$

Similarly to \cite{BBV19, BV19, BK20}, we obtain the following two cases:
\begin{itemize}
    \item \textit{Degenerate case}: $\vv_1(l_1, \la)$ and $\vv_2(l_2, \la)$ have common zeros, and the solution of IP~\ref{ip:But} is non-unique.
    \item \textit{Non-degenerate case}: $\vv_1(l_1, \la)$ and $\vv_2(l_2, \la)$ do not have common zeros, and the solution of IP~\ref{ip:But} is unique.
\end{itemize}

In the non-degenerate case, the solution of IP~\ref{ip:But} can be found by a constructive procedure similar to Algorithm~\ref{alg:easy}. Following the strategy of Section~\ref{sec:easy}, one can obtain the spectrum characterization for this case.

\smallskip

{\bf VII.} In the case of distinct reals $\{ h_j \}_{j = 1}^m$, Theorem~\ref{thm:char} can be derived from the results of Dobosevych and Hryniv~\cite{DH22}. Indeed, the boundary value problem~$L^*$ generates the operator
\begin{equation} \label{defB}
B = A + \langle ., \phi \rangle \psi,
\end{equation}
being a rank-one perturbation of the Laplacian $A u = [-u_j'']_{j = 1}^m$ in the Hilbert space $\mathcal H$ with the domain
$$
\mbox{dom}(A) = \{ u \in \mathcal H \colon u_j \in W_2^2[0, \pi], \, j = \overline{1, m}, \, u \: \textit{satisfies} \: \eqref{bcz}, \eqref{mcz} \}.
$$
In \eqref{defB}, $\phi$ is the Dirac $\de$-function corresponding to the central vertex of the graph, $\psi = [\overline{p_j}]_{j = 1}^m \in \mathcal H$. Clearly, the operator $A$ is self-adjoint. It can be shown that the eigenvalues of $A$ are simple and separated, so the theory of \cite{DH22} can be applied. However, in the case of non-real $\{ h_j \}_{j = 1}^m$, the operator $A$ is non-self-adjoint, so the method of \cite{DH22} does not work.

\section*{Appendix}

\setcounter{equation}{0}
\renewcommand\theequation{A.\arabic{equation}}

\setcounter{thm}{0}
\renewcommand\thethm{A.\arabic{thm}}

In the Appendix, we describe some physical processes modeled by the boundary value problems $L$ and $L^*$. The majority of the known mathematical models with integral BCs are constructed for differential operators on intervals. Here, we illustrate the idea of generalizing such models to the case of the star-shaped graph.

Firstly, we describe the model of a thermostat, following the ideas of \cite{Webb05, SD15}. Although in \cite{Webb05, SD15} nonlinear equations were studied, we consider the simplest linear heat equation
$$
\frac{\partial u}{\partial t} = a^2 \frac{\partial^2 u}{\partial x^2} , \quad t > 0, \quad x \in (0, 1),
$$
where $u = u(x, t)$ is a temperature of a heated bar at the point $x$ at the time $t$, $a > 0$ is a given constant. The BCs
$$
\frac{\partial u(x, t)}{\partial x}\bigg|_{x = 0} - h u(0, t) = 0, \quad \frac{\partial u(x, t)}{\partial x} \bigg|_{x = 1} + p u(\eta, t) = 0, \quad h, p \in \mathbb R, \quad \eta \in (0, 1],
$$
correspond to the heat exchange with the zero temperature environment at $x = 0$ and a controller at $x = 1$ adding or removing heat dependent on the temperature detected by a sensor at $x = \eta$. If there are several sensors at the points $\eta_1$, $\eta_2$, \ldots, $\eta_k$, then the right-hand BC takes the form
$$
\frac{\partial u(x, t)}{\partial x}\bigg|_{x = 1} + \sum_{j = 1}^k p_j u(\eta_j, t) = 0.
$$
Furthermore, if the sensors collect information on the temperature from the whole bar, then we obtain the integral condition
$$
\frac{\partial u(x, t)}{\partial x}\bigg|_{x = 1} + \int_0^1 p(x) u(x, t) \, dx = 0.
$$

Now consider $m$ bars of equal length $\pi$ connected all together as a star-shaped graph. The boundary value problem
\begin{gather} \nonumber
\frac{\partial u_j}{\partial t} = a^2 \frac{\partial^2 u_j}{\partial x_j^2} , \quad t > 0, \quad x_j \in (0, \pi), \quad j = \overline{1, m}, \\ \label{bc1}
\frac{\partial u_j(x_j, t)}{\partial x_j}\bigg|_{x_j = 0} - h_j u_j(0, t) = 0, \quad j = \overline{1, m}, \\ \label{mc1}
u_1(\pi, t) = u_j(\pi, t), \quad j = \overline{2, m}, \\ \label{mc2}
\sum_{j = 1}^m \left(\frac{\partial u_j(x_j, t)}{\partial x_j}\bigg|_{x_j = \pi} + \int_0^{\pi} p_j(x_j) u_j(x_j, t) \, dx_j \right) = 0
\end{gather}
models the heating process in this structure with a controller at the central vertex, corresponding to $x_j = \pi$. The controller adds or removes heat depending on the temperature on the whole graph. 

Similarly, a star-shaped graph of vibrating strings connected all together at the central vertex can be considered. Suppose that a controller is located 
at the central vertex.
If the influence of the controller
depends on the displacements of the whole graph, then such model is described by the equations
$$
\frac{\partial^2 u_j}{\partial t^2} = a^2 \frac{\partial^2 u_j}{\partial x_j^2} , \quad t > 0, \quad x_j \in (0, \pi), \quad j = \overline{1, m}, 
$$
with the conditions \eqref{bc1}-\eqref{mc2}. Here $u_j(x_j, t)$ is the lateral displacement of the $j$-th string at the point $x_j$ at the time $t$, \eqref{mc1} is the continuity condition at the central vertex, and the condition \eqref{mc2} describes the balance of forces at this vertex.

Obviously, the separation of variables in the both boundary value problems implies the eigenvalue problem $L$ of form \eqref{eqv}-\eqref{mcint}.

On the other hand, some physical models with frozen argument can be constructed by using the ideas of \cite{Krall75, BH21}.
The eigenvalue problem $L^*$ can be obtained by separating the variables in the boundary value problem for the following parabolic (hyperbolic) equations:
\begin{equation} \label{hp}
\frac{\partial^{\nu} u_j}{\partial t^{\nu}} = a^2 \frac{\partial^2 u_j}{\partial x_j^2} - \overline{p_j(x_j)} u_j(\pi, t)
\quad \nu = 1 \: (\nu = 2),
\quad t > 0, \quad x_j \in (0, \pi), \quad j = \overline{1, m},
\end{equation}
with the conditions \eqref{bc1}-\eqref{mc1} and
\begin{equation} \label{stmc}
\sum_{j = 1}^m \frac{\partial u_j(x_j, t)}{\partial x_j}\bigg|_{x_j = \pi} = 0.
\end{equation}

In the parabolic case, the system \eqref{hp}, \eqref{bc1}, \eqref{mc1}, \eqref{stmc} models the heat conduction in a star-shaped structure of rods possessing a sensor at the central vertex and an external distributed heat source. This source is described by the functions
\begin{equation} \label{deff}
f_j(x_j, t) = -\overline{p_j(x_j)} u_j(\pi, t), \quad j = \overline{1, m},
\end{equation}
on the edges of the graph, that is, the source power is proportional to the temperature at the central vertex measured by the sensor. In is shown in \cite{BH21} that such model can be implemented by electric conductive rods
of a constant thermal conductivity but possessing the variable electrical resistance
independent of the temperature.

In the hyperbolic case, the system \eqref{hp}, \eqref{bc1}, \eqref{mc1}, \eqref{stmc} corresponds to an oscillatory process under a damping external force $f_j(x, t)$ defined by \eqref{deff}. This force is proportional to the displacement at the internal point measured by a sensor. In particular, this model is relevant to vibrating wires affected by a magnetic field (see \cite{BH21} for details).
It is worth mentioning that the MC \eqref{stmc} corresponds to the law of energy conservation in the parabolic case and to the balance of forces at the central vertex in the hyperbolic case.

Thus, from the physical point of view, there are two principal situations:

\begin{enumerate}
    \item A controller located at one point is influenced by sensors that collect information from the whole structure. Then, the process is modeled by a boundary value problem with integral MCs.
    \item A system is influenced by a distributed force depending on measurements of a sensor located at one point. Such processes are described by functional-differential equations with frozen argument.
\end{enumerate}

From the mathematical point of view, these two situations correspond to mutually adjoint operators. The inverse spectral problem, which consists in the recovery of the unknown coefficients from the spectrum, corresponds to the construction of such physical systems with desired properties.

\bigskip

{\bf Funding.} This work was supported by Grant 21-71-10001 of the Russian Science Foundation, https://rscf.ru/en/project/21-71-10001/.

\medskip

{\bf Acknowledgement.} The author is grateful to Professors Sergey A. Buterin and Baltabek E. Kanguzhin for valuable discussions.

\noindent Natalia Pavlovna Bondarenko \\
1. Department of Applied Mathematics and Physics, Samara National Research University, \\
Moskovskoye Shosse 34, Samara 443086, Russia, \\
2. Department of Mechanics and Mathematics, Saratov State University, \\
Astrakhanskaya 83, Saratov 410012, Russia, \\
e-mail: {\it BondarenkoNP@info.sgu.ru}
\end{document}